\documentclass[12pt]{article}
\usepackage{latexsym}
\usepackage{amsfonts}
\usepackage{amscd}
\usepackage{amsmath, amsthm, amssymb}
\usepackage[all]{xy}
\newtheorem{thm}{Theorem}[section]
\newtheorem{cor}[thm]{Corollary}
\newtheorem{lem}[thm]{Lemma}
\newtheorem{prop}[thm]{Proposition}
\newcommand{\rk}{\mathrm{rank}}

\newcommand{\lie}[1]{\mathfrak{#1}}

\newcommand{\Z}{\mathbb{Z}}
\newcommand{\R}{\mathbb{R}}
\newcommand{\C}{\mathbb{C}}
\newcommand{\Q}{\mathbb{Q}}

\newcommand{\co}{:}
\newcommand{\ignore}[1]{}
\theoremstyle{definition}

\newtheorem{define}{Definition}
\newtheorem{rmk}{Remark}
\newtheorem{example}{Example}

\title{Classifying spaces of twisted loop groups}
\author{Thomas J. Baird}
\date{}
\begin{document}
\maketitle

\begin{abstract}
We study the classifying space of a twisted loop group $L_{\sigma}G$ where $G$ is a compact Lie group and $\sigma$ is an automorphism of $G$ of finite order modulo inner automorphisms. Equivalently, we study the  $\sigma$-twisted adjoint action of $G$ on itself. We derive a formula for the cohomology ring $H^*(BL_{\sigma}G)$ and explicitly carry out the calculation for all automorphisms of simple Lie groups. More generally, we derive a formula for the equivariant cohomology of compact Lie group actions with constant rank stabilizers.
\end{abstract}

\section{Introduction}

Let $G$ be a compact, connected Lie group and let $\sigma \in Aut(G)$ be an automorphism of $G$.  The \emph{twisted loop group} $ L_{\sigma}G$ is the topological group $ L_{\sigma}G$ of continuous paths $\gamma: I \rightarrow G$ satisfying $\gamma(1) = \sigma(\gamma(0))$, with point wise multiplication and compact-open topology. In the special case that $\sigma$ is the identity automorphism, $L_{\sigma}G$ is the usual (continuous) loop group $LG$.\footnote{We work with continuous loops throughout, but by work of Palais (\cite{p} Thm 13.14) the homotopy type of $BL_{\sigma}G$ is unchanged if we work for $C^k$ loops for $k\geq 0$  or $L_p^k$ loops for $k > 1/p$}   The main result  of this paper is a formula for the cohomology ring of the classifying space $H^*(BL_{\sigma}G)$.

The isomorphism type of $L_{\sigma}G$ depends only on the outer automorphism $[\sigma] \in Out(G) = Aut(G)/Inn(G)$ represented by $\sigma$ (see \S 4). If $G$ is semisimple, then the outer automorphism group $Out(G)$ is naturally isomorphic to the automorphism group of the Dynkin diagram of $G$, so $Out(G)$ is finite.

\begin{thm}\label{thm0}
Let $G$ be a semisimple compact, connected Lie group with Weyl group $W$ and let  $\sigma \in Aut(G)$ be an automorphism with corresponding outer automorphism $[\sigma] \in Out(G)$.  Let $G^{\sigma}$ denote the subgroup of elements fixed by $\sigma$, with identity component $G^{\sigma}_0$.  Then the inclusion of $G^{\sigma}_0 \subseteq G$ induces an injection in cohomology
$$ H^*(BL_{\sigma} G; F) \hookrightarrow H^*(BL G^{\sigma}_0;F) $$
for coefficient fields $F$ of characteristic coprime to the order of $W$, the order of $[\sigma]$, and to the number of path components of $G^{\sigma}$. The image of the injection is the ring of invariants $$ H^*(BL_{\sigma} G; F) \cong H^*(BL G^{\sigma}_0;F)^{W_{\sigma}}$$
under an action by a subgroup $W_{\sigma} \subseteq W$. Specifically, $W_{\sigma} = N_G(T^{\sigma})/T$ is the quotient of the normalizer of a maximal torus $T^{\sigma} \subseteq G_0^{\sigma}$ by a maximal torus $T \subseteq G$.  
\end{thm}
In many cases (see \S \ref{Simplifying}), $W_{\sigma}$ acts via outer automorphisms of $G_0^{\sigma}$, which are well understand. Classifying spaces of untwisted loop groups are also well understood (see Proposition \ref{untwistedcase}) so Theorem \ref{thm0} enables explicit calculation of $H^*(BL_{\sigma} G; F)$ whenever the hypotheses hold. We carry out this calculation for all automorphisms of compact, connected simple Lie groups in \S \ref{examplessect}.

More generally, we derive a formula (Proposition \ref{prop1}) for $H^*(BL_{\sigma} G; F)$ if $G$ is compact, connected and $\sigma$ is conjugate to an automorphism of finite order. The proof uses the following model of $BL_{\sigma} G$. Denote by $G_{Ad_{\sigma}}$ the left $G$-space whose underlying space is the group manifold $G$ and with \emph{twisted adjoint action}
$$ Ad_{\sigma} :G \times G_{Ad_{\sigma}} \rightarrow G_{Ad_{\sigma}} , ~~~~Ad_{\sigma}(g)(x) = gx \sigma(g^{-1}) .$$
The classifying space $BL_{\sigma}G$ is homotopy equivalent to the homotopy quotient $ EG \times_G G_{Ad_{\sigma}}$ (Lemma \ref{othermodel}).

If $\sigma$ has order $n$, then form the compact semi-direct product $\Z_n \ltimes G $ by the rule $( a, g) \cdot (b, h) = (a+b, \sigma^{-b}(g) h)$. The $G$-space $G_{Ad_{\sigma}}$ is identified with the standard adjoint action of $G \cong \{0\} \times G$ on the path component  $ \{1 \} \times G$.  By a result of de Siebenthal (see \cite{des}, last theorem in chapter II), the stabilizers $G_p$ of this action all have the same rank; this permits us to apply the following result which may be of more general interest.

\begin{thm}\label{thm2}
Let $G$ be a connected, compact Lie group and $X$ a connected, compact Hausdorff $G$-space with constant rank stabilizers.  Choose any $p \in X$ and let $T_p \subseteq G$ be a maximal torus in the stabilizer $G_p$ of $p$.  Then the inclusions $N_G(T_p) \subseteq G$ and $X^{T_p} \subseteq X$ induce an isomorphism in equivariant cohomology  $$ H^*_G(X) \cong H^*_{N_G(T_p)}(X^{T_p}) $$ for coefficient fields of characteristic coprime to order of the Weyl group of $G$.
\end{thm}

The proof of Theorem \ref{thm2} is a straightforward generalization of the special case proven in Baird (\cite{b2} Thm 3.3) where the stabilizers $G_p$ were assumed to have rank equal to that of $G$. 

Twisted loop groups have been studied in relation to the representation theory of affine Lie algebras  \cite{ps, mw, w}  and to Wess-Zumino-Witten theory \cite{st}.  A special class of twisted loop groups - called \emph{real loop groups} - was introduced in \cite{b3} in the course of calculating $\Z_2$-Betti numbers of moduli spaces of real vector bundles over a real curve.  The results of the current paper will be applied in future work to study the cohomology of these moduli spaces in odd characteristic. 

\textbf{Notation:} Given a topological group $G$ and $G$-space $X$, we denote by $X_{hG}$ or $EG \times_G X$ the Borel construction homotopy quotient.

\textbf{Acknowledgements:} I would like to thank Misha Kotchetov for helpful advice about triality and the referee for pointing towards a gap in a proof appearing in an earlier draft.  This research was supported by an NSERC Discovery Grant.

\section{Cohomological Principal Bundles}

We recall some background material from \cite{b2}. Let $ f: X \rightarrow Y $ be a continuous map between topological
spaces $X$ and $Y$, and let $ \Gamma$ be a topological group
acting freely on $X$, such that $ X \rightarrow X /
\Gamma $ is a principal bundle.

\begin{define}\label{princ}
We say $(f\co X  \rightarrow Y, \Gamma )$ is a cohomological
principal bundle for the cohomology theory $H^*$ if:
\\
i) $f$ is a closed surjection
\\
ii) $f$ descends through the quotient to a map $h$,
\\
\begin{equation}\begin{CD}
\xymatrix{ X \ar[d]^{\pi} \ar[dr]^f \\
       X/\Gamma \ar[r]^h & Y}
\end{CD}\end{equation}
\\
iii) $H^*(h^{-1}(y)) \cong H^*(pt)$ for all $y \in Y$
\end{define}

Let $H^*(X;F)$ denote sheaf cohomology of the constant sheaf $F_X$, where $F$ is a field. The following is a simplification of  (\cite{b2} Corollary 2.4). 

\begin{prop}\label{corollary}Let $\Gamma$ be a compact Lie group and let $(f\co X \rightarrow Y, \Gamma )$ be a cohomological principal bundle for $H^*(.,F)$, where $X$ is a paracompact Hausdorff space. Then
$H^*(Y;F) \cong H^*(X/\Gamma;F)$.
\end{prop}

We also make repeated use of the following from Bredon \cite{br} 19.2.

\begin{thm}\label{cov} Let $X$ be a topological space, let $\Gamma$ be a finite group acting on $X$ and let $\pi\co X \rightarrow X/\Gamma$ denote the quotient map onto the orbit space $X/\Gamma$.  If $F$ is a field satisfying $gcd(char(F),\#\Gamma)=1$, then

\begin{equation}
\pi^*\co H(X/\Gamma;F) \rightarrow H(X;F)^{\Gamma}
\end{equation}
is an isomorphism, where $H(X;F)^{\Gamma}$ denotes the ring of $\Gamma$ invariants.
\end{thm}

A particular example of Theorem \ref{cov} is that $B\Gamma$ is acyclic for coefficient field $F$ coprime to the order of $\Gamma$, because $H^*(B\Gamma;F) = H^*(E\Gamma /\Gamma;F) = H^*(E\Gamma;F)^{\Gamma}$ and $E\Gamma$ is acyclic because it is contractible.

\section{Cohomology of $G$-spaces with constant rank stabilizers}

Let $G$ be a compact Lie group and $X$ a left $G$-space which is compact, connected and Hausdorff. Denote by $G_x$ the stabilizer of a point $x \in X$ and by $G_x^0$ the identity component of $G_x$. For a given point $p \in X$, let $T_p$ denote a maximal torus in the identity component $G_p^0$ of $G_p$.  Define an equivariant map 
\begin{equation}\label{maptothe}
 \phi: G \times X^{T_p} \rightarrow X,~~~~\phi((g,x)) = g \cdot x 
\end{equation}
where $G$ acts on $G \times X^{T_p}$ by $g \cdot (h,x) =
(gh,x)$. The normalizer $N(T_p) = N_G(T_p)$, acts on $G \times X^{T_p}$ from the right by
\begin{equation}
(g,x) \cdot n = (gn, n^{-1} \cdot x),
\end{equation}
leaving $\phi$ invariant and commuting with the $G$ action.

\begin{prop}\label{prop0} Under the hypotheses of Theorem \ref{thm2}, the pair  $( \phi: G \times X^{T_p} \rightarrow X, N(T_p))$ is a
cohomological principal bundle. 
\end{prop}
We begin with a few of lemmas.

\begin{lem}\label{lem0}
Under the hypotheses of Theorem \ref{thm2}, given any two points $x,y \in X$,  the maximal tori $T_x \subseteq G_x^0$ and $T_y \subseteq G_y^0$ are conjugate in $G$.
\end{lem}

\begin{proof}
Given $p \in X$, the set  $A = \{ x \in X~|~T_x \text{ is conjugate to $T_p$ }\}$ is equal to the image of (\ref{maptothe}). The fixed point set $X^{T_p}$ is closed in $X$, hence compact. Since $G$ is compact, the product $G \times X^{T_p}$ is compact, and the image of (\ref{maptothe}) is compact, hence closed.  Thus $A$ is a closed subset of $X$.

Since $X$ is compact and Hausdorff it is completely regular.  By a theorem of Gleason (Theorem 3.3 of \cite{gl}), $G$-orbits in $X$ admit local cross sections. In particular, for every $x \in X$ there is an open neighbourhood $x \in U \subseteq X$ such that for every $y \in U$, the stabilizer $G_y$ is a subgroup of a conjugate of $G_x$.  Since $G_x$ and $G_y$ have the same rank, this implies that $T_x$ and $T_y$ are conjugate.  It follows that $A$ is an open subset of $X$. Since $X$ is connected, it follows that $A = X$.
\end{proof}

\begin{lem}\label{lem1}
Let $G$ act on $X$ from the left and let $x \in X^{T_p}$.  Then $g \cdot x
\in X^{T_p}$ if and only if $g \in N(T_p)G_x^0$, where $G_x^0$ is the
identity component of the stabilizer $G_x$.
\end{lem}

\begin{proof}
If $g \cdot x \in X^{T_p}$, then $g^{-1} t g \cdot x = x$ for all $t \in
T_p$, so

\begin{equation}
 g^{-1}T_pg \subset G_x.
\end{equation}
Since stabilizers have constant rank, both $T_p$ and $g^{-1}T_pg$ are maximal in $G_x$, so for
some $h \in G_x^0$, $h^{-1}g^{-1}T_pgh = T_p$, and thus $g \in N(T_p)G_x^0$.
The other direction is clear.
\end{proof}

\begin{lem}\label{lem2}
Let $( \phi: G \times X^{T_p} \rightarrow X, N(T_p))$ be defined as above.
For every $x \in X$, the orbit space $\phi^{-1}(x)/N(T_p) \cong G^0_x / N_{G^0_x}(H)$,
where $H$ is a maximal torus in $G^0_x$.
\end{lem}

\begin{proof}
We may assume by equivariance that $x \in X^T$. We have isomorphisms of right $N(T_p)$-spaces
\begin{equation}
\phi^{-1}(x) = \{ (g,y) \in G \times X^{T_p} | g \cdot y = x \} = \{(g^{-1}, g \cdot x) | g \in N(T_p) G_x^0 \} \cong G_x^0N(T_p)
\end{equation}
where the middle equality follows from the preceding lemma.  It
follows that

\begin{equation}
\phi^{-1}(x)/N(T_p) \cong G^0_xN(T_p)/N(T_p) \cong G_x^0/N_{G_x^0}(T_p).
\end{equation}
\end{proof}

\begin{proof}[Proof of Proposition \ref{prop0}]
Since both $G$ and $X$ are compact, it follows that $G \times X^{T_p}$ is compact and thus $\phi$ is closed.  From Lemma \ref{lem0}, it follows that every $G$-orbit in $X$ must intersect $X^{T_p}$, so $\phi$ is surjective. Finally, the homeomorphism $\phi^{-1}(x)/N(T_p) \cong G^0_x / N_{G^0_x}(H)$ from Lemma \ref{lem2} implies that $H^*(\phi^{-1}(x)/N(T_p);F) $ is acyclic over fields of characteristic coprime to the order of the Weyl group (as explained in \cite{b2} \S 3).   
\end{proof}

\begin{proof}[Proof of Theorem \ref{thm2}]

The map $G \times_{N(T_p)} X^{T_p} \rightarrow X$ is $G$-equivariant and a cohomology isomorphism, so it induces an isomorphism in equivariant cohomology
$$ H^*_G(X) \cong H^*_G( G \times_{N(T_p)} X^{T_p}).$$
The action of $N(T_p)$ on $G \times X^{T_p}$ is free and commutes with $G$, so we also have an isomorphism
$$H^*_G( G \times_{N(T_p)} X^{T_p}) \cong   H^*_{G \times N(T_p)} ( G \times X^{T_p}).$$
Finally, $G$ acts freely on $G \times X^{T_p}$, so we have an isomorphism
$$H^*_{G \times N(T_p)} ( G \times X^{T_p}) \cong H^*_{N(T_p)} (X^{T_p}). $$

\end{proof}

\section{Formula for compact, connected $G$}

Let $G$ be a compact, connected Lie group and $\sigma \in Aut(G)$. Let $G^{\sigma} \leq G$ denote the subgroup of elements fixed by $\sigma$, and let $T^{\sigma}$ be a maximal torus in (the identity component of) $G^{\sigma}$. Let $C(T^{\sigma})  = C_G(T^{\sigma})$ and $N(T^{\sigma}) =N_G(T^{\sigma}) $ be the centralizer and normalizer of $T^{\sigma}$ in $G$ respectively. The twisted adjoint action restricts to an action of $N(T^{\sigma})$ on $C(T^{\sigma})$, which we denote by $C(T^{\sigma})_{Ad_{\sigma}}$. The goal of this section is to prove  

\begin{prop}\label{prop1}
Let $G$ be a connected, compact Lie group and $\sigma \in Aut(G)$ an automorphism such that some conjugate $g \sigma g^{-1}$ has finite order. Then there is a cohomology isomorphism
$$ H^*(BL_{\sigma}G) \cong H^*_{N(T^{\sigma})}(C(T^{\sigma})_{Ad_{\sigma}}) $$  for coefficient fields coprime to the order of the Weyl group of $G$.
\end{prop}

The following result is not original (it follows implicitly from \cite{w} ), but I have not been able to find a clean statement in the literature.

\begin{lem}\label{othermodel}
 There is a natural homotopy equivalence $ BL_{\sigma}G \cong EG \times_G G_{Ad_{\sigma}}$.
\end{lem}

\begin{proof}
Consider the action of $L_{\sigma}G$ on the contractible based path space $$PG := Maps((I,0), (G, Id_G))$$ by $ (\gamma \cdot x) (t) = \gamma(t) x \gamma(0)^{-1}$.  Since $PG$ is contractible, the homotopy quotient $PG_{hL_{\sigma}G}$ is a model for $BL_{\sigma}G$.

The based loop group $\Omega G := \{ \gamma \in L_{\sigma}G~|~\gamma(0)=\gamma(1) = Id_G\}$  acts freely on $PG$, so $PG_{hL_{\sigma}G}$ is equivalent to the homotopy quotient of the residual action of $L_{\sigma}G/\Omega G \cong G$ on  $PG/\Omega G \cong G_{Ad_{\sigma}}$.
\end{proof}

The isomorphism class of $L_{\sigma}G$ is depends only on the element of the outer automorphism group $Out(G) = Aut(G)/Inn(A)$ represented by $\sigma$ (see  \cite{ps} section 3.7).  Similarly, if $\sigma' = Ad_h \circ \sigma$ represent the same outer automorphism, then the map $$G_{Ad_{\sigma}} \rightarrow G_{Ad_{\sigma'}},~~~~ x \mapsto hxh^{-1}$$ in an isomorphism of $G$-spaces. Thus we may assume without loss of generality that $\sigma$ has finite order.

\begin{lem}\label{lem4}
Let $G$ be a compact, connected Lie group. If $\sigma \in Aut(G)$ has finite order then the twisted adjoint action of $G$ on $G_{Ad_{\sigma}}$ has constant rank stabilizers.
\end{lem}

\begin{proof}
If $\sigma$ has order $n$, then it can be used to construct a semi-direct product $ \Z_n \ltimes G$.  It is explained in the introduction that the action of $G$ on $G_{Ad_{\sigma}}$ is isomorphic to standard adjoint action of $G =\{0\}\times G$ on the path component $\{1\}\times G \cong G_{Ad_{\sigma}}$.	
	
The result now follows from a fundamental property of the adjoint action of disconnected compact Lie groups found at the end chapter II of \cite{des}.
\end{proof}

\begin{proof}[Proof of Proposition \ref{prop1}]
By Proposition \ref{othermodel} we have $H^*(BL_{\sigma}G) \cong H^*_G(G_{Ad_{\sigma}})$.  By Lemma \ref{lem4} and Theorem \ref{thm2} we have
$$ H^*_G(G_{Ad_{\sigma}})  \cong H^*_{N(T_p)} (G_{Ad_{\sigma}}^{T_p}) $$
for any choice of $p \in G_{Ad_{\sigma}}$.  Choose $p = Id_G$.

The stabilizer of the identity element $Id_G \in G$ under the twisted adjoint action is exactly the subgroup $G^{\sigma}$ of elements invariant under $\sigma$.  Let $T^{\sigma}$ denote a maximal torus of $G^{\sigma}$. The restriction of the twisted adjoint action to $T^{\sigma}$ agrees with the ordinary adjoint action.  It follows that the set of $T^{\sigma}$-fixed points is precisely the centralizer  $C(T^{\sigma}) := \{ g \in G ~|~g t = t g, \text{ for all } t \in T^{\sigma} \}$. Thus
$$ H^*_{N(T_p)} (G_{Ad_{\sigma}}^{T_p})  \cong H^*_{N(T^{\sigma})}(C(T^{\sigma})_{Ad_{\sigma}}) $$ as desired.
\end{proof}

\begin{example} If $\sigma \in Aut(G)$ is the identity, we have $T^{\sigma}=T$ is a maximal torus with $N(T^{\sigma}) = N(T)$  acting on $ C(T^{\sigma}) = T$ by the standard adjoint action. Proposition \ref{prop1} gives us the formula

\begin{equation}\label{familiar}
 H^*(BLG) = H^*_G(G) \cong H^*_{N(T)}(T) \cong (H^*(T) \otimes H^*(BT))^W
\end{equation}
for coefficients coprime to the order of the Weyl group $W =N(T)/T$.
\end{example}

\begin{rmk}
If $G$ is abelian, then $N(T^{\sigma}) = C(T^{\sigma}) = G$,  so Proposition \ref{prop1} offers no improvement over Lemma \ref{othermodel}.  The formula is more interesting in the opposite extreme when $G$ is semisimple, which is the subject of Theorem \ref{thm0}.
\end{rmk}

\section{Proof of Theorem \ref{thm0}}

Assume throughout this section that $G$ is a compact, connected, semisimple Lie group, and that cohomology is taken with coefficient field $F$ of characteristic $p$ coprime to the orders of the Weyl group $W_G = N_G(T)/T$, of $[\sigma]$, and of $\pi_0(G^{\sigma})$.

\begin{lem}\label{a torus}
Let $T^{\sigma}$ be a maximal torus in $G^{\sigma}_0$ and $C(T^{\sigma})$ the centralizer of $T^{\sigma}$ in $G$. Then,
\begin{itemize}
\item[(a)]  $C(T^{\sigma}) = T$ is a maximal torus in $G$ and $T^{\sigma} = G^{\sigma}_0 \cap T$.
\item[(b)] The restriction of $\sigma$ to $T$ preserves a Weyl chamber, and thus has finite order equal to that of the outer automorphism $[\sigma] \in Out(G)$ .
\end{itemize}
\end{lem}

\begin{proof}
This is mostly just a restatement [\cite{des}, chapter II, \S 3 Proposition 2].  The only addition is that $\sigma_T$ has order equal to $[\sigma]$.  This follows because if $[\sigma]$ has order $n$, then $\sigma^n$ is an inner automorphism that preserves a Weyl chamber of $T$ and thus must restrict to the identity map on $T$.
\end{proof}

The twisted adjoint action restricts to the standard adjoint action for the subgroup $ N_{G_0^{\sigma}}(T^{\sigma}) \subseteq N_G(T^{\sigma}) $ acting on the subspace $T^{\sigma} \subseteq T$, so these inclusions give rise to morphism in equivariant cohomology, which by Proposition \ref{prop1} fits into a commutative diagram

\begin{equation}\begin{CD}\label{thefirst}
\xymatrix{  H^*(BL_{\sigma}G) \ar[d]^{\cong} \ar[r] & H^*(BLG_0^{\sigma}) \ar[d]^{\cong}   \\
       H^*_{N_G(T^{\sigma})}(T_{Ad_{\sigma}})\ar[r] &  H^*_{N_{G_0^{\sigma}}(T^{\sigma})}(T^{\sigma}_{Ad})}
\end{CD}\end{equation}
where the top arrow is the subject of Theorem \ref{thm0}.

Since $C(T^{\sigma}) \le N_G( T^{\sigma})  \le N_G(C(T^{\sigma}))$, it follows from Lemma \ref{a torus} that $$T \le N_G(T^{\sigma}) \le N_G(T).$$
Define $ W_{\sigma} := N_G(T^{\sigma})/T$. Then $W_{\sigma} \subseteq W_G = N_G(T)/T$ is a finite group of order coprime to $p$. By Theorem \ref{cov}, there is a natural isomorphism 
$$H^*_{N_G(T^{\sigma})}(T_{Ad_{\sigma}}) \cong H^*_T(T_{Ad_{\sigma}})^{W_{\sigma}}.$$

The Weyl group $W_{G_0^{\sigma}} := N_{G_0^{\sigma}}(T^{\sigma})/T^{\sigma}$ acts faithfully on $T^{\sigma}$ by the adjoint action, so the inclusion $ N_{G_0^{\sigma}}(T^{\sigma}) \hookrightarrow N_{G}(T^{\sigma})$ descends to an injection $W_{G_0^{\sigma}} \hookrightarrow W_{\sigma}$. It follows from Lagrange's Theorem that the order of $W_{G_0^{\sigma}}$ is coprime to $p$. We gain a natural isomorphism
$$ H^*_{N_{G_0^{\sigma}}(T^{\sigma})}(T^{\sigma}_{Ad}) \cong H^*_{T^{\sigma}}(T^{\sigma}_{Ad})^{W_{G_0^{\sigma}}} . $$ 

These natural isomorphisms permit us to replace (\ref{thefirst}) with the commuting diagram

\begin{equation}\begin{CD}\label{commagain}
\xymatrix{ H^*(BL_{\sigma}G) \ar[r] \ar[d]^{\cong} &  H^*(BLG_0^{\sigma})  \ar[d]^{\cong}\\
H^*_T( T_{Ad_{\sigma}})^{W_{\sigma}} \ar[r] & H^*_{T^{\sigma}}(T^{\sigma}_{Ad})^{W_{G_0^{\sigma}}}}
\end{CD}\end{equation}

\begin{lem}\label{alema}
The inclusion $T^{\sigma} \subseteq T$ induces an isomorphism in equivariant cohomology
$$ H^*_T( T_{Ad_{\sigma}}) \cong H^*_{T^{\sigma}}(T^{\sigma}_{Ad}) .$$
\end{lem}

\begin{proof}
The action of $T^{\sigma}$ on $T^{\sigma}_{Ad}$ is trivial, so the homotopy quotient is the product $BT^{\sigma} \times T^{\sigma}$ and the equivariant cohomology ring is $$H^*_{T^{\sigma}}(T^{\sigma}_{Ad}) \cong H^*(BT^{\sigma}) \otimes H^*(T^{\sigma}).$$
The twisted adjoint action of $t \in T$ on $x \in T_{Ad_{\sigma}}$ 
$$ Ad_{\sigma}(t)(x) = t x \sigma(t)^{-1} = t \sigma(t)^{-1} x $$
is simply translation by $t \sigma(t)^{-1}$.  Consequently, $T$ acts on $T_{Ad_{\sigma}}$ with constant stabilizer $G^{\sigma}\cap T$.

Choose a complementary subtorus $T'$ so that  $T = T^{\sigma} \times T'$.  Then the factor $T^{\sigma}$ acts trivially on $T_{Ad_{\sigma}}$, so the homotopy quotient satisfies
$$ (T_{Ad_{\sigma}})_{h T} \cong BT^{\sigma} \times  (T_{Ad_{\sigma}})_{h T'}  $$
where in the second factor we consider the restricted action of $T'$ on $T_{Ad_{\sigma}}$ which has constant stabilizer $T'\cap G^{\sigma}$. It follows that the projection map onto the orbit space
$$ (T_{Ad_{\sigma}})_{h T'} \rightarrow T_{Ad_{\sigma}}/T'= T_{Ad_{\sigma}}/T $$
has homotopy fibre $B(T' \cap G^{\sigma})$.

Observe that the induced homomorphism $ T' \cap G^{\sigma} \rightarrow \pi_0(G^{\sigma})$ is injective, because $T' \cap G^{\sigma}_0 = T' \cap T^{\sigma} = \{Id_G\}$.  In particular, the order of $T' \cap G^{\sigma}$  divides the order of $\pi_0(G^{\sigma})$, so $B(T' \cap G^{\sigma})$ is acyclic over the field $F$ and 
$$ H^*_T(T_{Ad_{\sigma}}) \cong H^*(BT^{\sigma}) \otimes H^*_{T'}(T_{Ad_{\sigma}}) \cong  H^*(BT^{\sigma}) \otimes H^*(T_{Ad_{\sigma}}/T).$$
The result now follows from Lemma \ref{covemap}.

\ignore{it only remains to prove that the map $\phi: T^{\sigma} \rightarrow T_{Ad}/T$. (the composition of inclusion $T^{\sigma} \hookrightarrow T$, with the quotient map) in a cohomology isomorphism. Let $H := \{ t \sigma(t)^{-1}| t \in T\}$. Then $\phi$ can be identified with the corresponding group homomorphism $$\phi': T^{\sigma} \rightarrow T/H.$$  
Since $\phi'$ is a homomorphism between tori of equal rank, it induces a cohomology isomorphism if and only if $\ker(\phi') = T^{\sigma} \cap H$ is finite of order coprime to $char(F)$.

Suppose that  $ t \in H \cap T^{\sigma}$.  Then $t = \sigma(t)$ and  $t = s \sigma(s)^{-1}$ for some $s \in T$. If $\sigma$ has order $n$, then
$$  t^n =  t \sigma(t) \sigma^2(t) ... \sigma^{n-1}(t) =  (s \sigma(s)^{-1})(\sigma(s) \sigma^2(s)^{-1}) ...( \sigma^{n-1}(s) \sigma^n(s)^{-1})  = Id_T.$$
Thus $\ker(\phi')$ is a subgroup of $T_n := \{ t \in T| t^n = Id_T\}$ which is a group of order $n^{\rk(T)}$. Since $n$ is coprime to $char(F)$, the result now follows by Lagrange's Theorem.}
\end{proof}

\begin{lem}\label{covemap}
The map $\phi: T^{\sigma} \rightarrow T_{Ad_{\sigma}}/ T$ obtained by composing inclusion and quotient maps is a covering map of finite degree coprime to $p$. In particular, $\phi$ induces a cohomology isomorphism in characteristic $p$.
\end{lem}

\begin{proof}
Let $n$ be the order of $\sigma|_T$. By Lemma \ref{a torus} $n$ is coprime to $p$.	
	
As explained in the proof of Lemma \ref{alema}, the twisted adjoint action of $t \in T$ on $T_{Ad_{\sigma}}$ 
is simply translation by $t \sigma(t)^{-1}$.  The orbit space $T_{Ad_{\sigma}}/T$ may thus be identified with the coset space $T/H$ where $H := \{ t \sigma(t)^{-1}| t \in T\}$, and $\phi$ can be identified with the corresponding group homomorphism $$\phi': T^{\sigma} \rightarrow T/H.$$   Since $\phi'$ is a homomorphism between tori of equal rank, it is enough to show that  $\ker(\phi') = T^{\sigma} \cap H$ has finite order dividing a power of $n$.

Suppose that  $ t \in H \cap T^{\sigma}$.  Then both $t = \sigma(t)$ and  $t = s \sigma(s)^{-1}$ for some $s \in T$. Thus
$$  t^n =  t \sigma(t) \sigma^2(t) ... \sigma^{n-1}(t) =  (s \sigma(s)^{-1}) \dots( \sigma^{n-1}(s) \sigma^n(s)^{-1})  = s \sigma^n(s)^{-1} = Id_T$$
so $\ker(\phi')$ is a subgroup of $T_n := \{ t \in T| t^n = Id_T\}$ which is a group of order $n^{\rk(T)}$. The result follows by Lagrange's Theorem. 
\end{proof}

Next, we want to understand the $W_{\sigma}$ and $W_{G_0^{\sigma}}$ actions.  Observe that the residual $W_{G_0^{\sigma}}$-action on the homotopy quotient $(T_{Ad})_{hT^{\sigma}} = BT^{\sigma} \times T^{\sigma}$ acts diagonally in the standard way on each factor, so the action extends to $W_{\sigma} = N_G(T^{\sigma})/T$ in the standard way.

\begin{lem}\label{gpat}
The isomorphism $ H^*_{T^{\sigma}}(T^{\sigma}_{Ad}) \cong H^*_T( T_{Ad_{\sigma}})$ defined in Lemma \ref{alema} is $W_{\sigma}$-equivariant with respect to the actions defined above. 
\end{lem}

\begin{proof}

Both actions are diagonal with respect to the Kunneth factorizations defined in the proof of Lemma \ref{alema}:
\begin{eqnarray*} H_{T^{\sigma}}^*(T^{\sigma}_{Ad}) & \cong & H^*(BT^{\sigma})\otimes H^*(T^{\sigma})\\
H^*_T( T_{Ad_{\sigma}}) &\cong & H^*(BT^{\sigma})\otimes H^*(T_{Ad_{\sigma}}/T).
\end{eqnarray*}

The action on the first factors are the same, since the group $T^{\sigma}$ acts trivially in both cases. It remains to consider action on the second factors are equivariant with respect to the the isomorphism $H^*(T^{\sigma}) \cong H^*(T_{Ad_{\sigma}}/T)$ from Lemma \ref{covemap}.

Let $\phi: T^{\sigma}_{Ad} \rightarrow T_{Ad_{\sigma}}/T$ be the covering map,  $n \in N_G(T^{\sigma})$, and $x \in T^{\sigma}_{Ad}$. Then

$$ \phi \circ Ad(n)(x) = \phi(nxn^{-1}) = [n x n^{-1}] $$
while
$$Ad_{\sigma}(n) \circ \phi(x) = [ n x \sigma(n^{-1})]= [ n x n^{-1} n \sigma(n^{-1})] . $$
The two maps $\phi \circ Ad(n)$ and $Ad_{\sigma}(n) \circ \phi$ agree up to translation by $n \sigma(n^{-1}) \in T$, so they are homotopic and define the same map on cohomology.
\end{proof}

\begin{proof}[Proof of Theorem \ref{thm0}]

Consider again the diagram (\ref{commagain}). It follows from Lemma \ref{alema} that the horizontal arrows are injective, and from Lemma \ref{gpat} that the image of the bottom arrow is equal to  

$$  H^*_{T^{\sigma}}(T^{\sigma}_{Ad})^{W_{\sigma}} \subseteq H^*_{T^{\sigma}}(T^{\sigma}_{Ad})^{W_{G_0^{\sigma}}}$$

\end{proof}

\section{Simplifying the calculations}\label{Simplifying}

It can be tricky to apply Theorem \ref{thm0} directly, because it requires an explicit understanding of $W_{\sigma}$ and its action on $H^*(BLG_0^{\sigma}) $. Fortunately, matters simplify under certain conditions.  

Throughout this section let $G$ be a compact, connected Lie group, let $\sigma \in Aut(G)$ be an automorphism, and let $T \leq G$ be a maximal torus containing a maximal torus $ T^{\sigma} \leq G^{\sigma}_0$. There is a natural action of $Aut(G_0^{\sigma})$ on $BLG_0^{\sigma}$. By a result of Segal (\cite{s} section 3), inner automophisms act by isotopy, so we obtain a natural action of $Out(G_0^{\sigma})$ on $H^*(BLG_0^{\sigma})$.

\begin{prop}\label{blah0}
Suppose that the adjoint action of $W_{\sigma}$ on $T^{\sigma}$ consists of transformations that extend to automorphisms of $G_0^{\sigma}$. Then the $W_{\sigma}$-action on $H^*(BLG_0^{\sigma})$ factors through the $Out(G_0^{\sigma})$-action.
\end{prop}

\begin{proof}
The injection $H^*(BLG_0^{\sigma}) \hookrightarrow H^*_{T^{\sigma}}(T^{\sigma})$ in (\ref{commagain}) is equivariant with respect to automorphisms of $G_0^{\sigma}$ which preserve $T^{\sigma}$. Thus if every transformation of $T^{\sigma}$ extends, it follows that $$ H^*(BLG_0^{\sigma})^{W_{\sigma}} = H^*(BLG_0^{\sigma})^{\Gamma}$$
for some subset $\Gamma \subseteq Out(G_0^{\sigma})$. An automorphism of a maximal torus extends to at most one outer automorphism of its connected, compact Lie group, so $\Gamma$ is the image of a well-defined homomorphism $W_{\sigma} \rightarrow Out(G_0^{\sigma})$.
\end{proof}

\begin{rmk}
The $W_{\sigma}$-action on $T^{\sigma}$ doesn't always extend to automorphisms of $G_0^{\sigma}$. For example, let $G= SU(3)$ and $\sigma = Ad_{g}$ where $$ g =  \left( \begin{array}{ccc}
1 & 0 & 0 \\
0 & -1 & 0 \\
0 & 0 & -1 \end{array} \right) .$$ Then  $T^{\sigma} = T$, so $W_{\sigma} = W \cong S_3$, but the $W$-action does not extend to $G^{\sigma} \cong U(2)$ because it does not preserve its root system.
\end{rmk}

We can use root systems to check whether the $W_{\sigma}$-action extends to $G_0^{\sigma}$. Let $\lie{g},$ $\lie{g}^{\sigma}$, $\lie{t}$, and $\lie{t}^{\sigma}$ denote the complexified Lie algebras of $G$, $G_0^{\sigma}$, $T$, and $T^{\sigma}$ respectively. The root system $\Phi \subset \lie{t}^*$ is simply the set of weights of the $\lie{t}$-module $\lie{g}/\lie{t}$ under the adjoint action. Similarly, the root system of $\Phi_{\sigma} \subset (\lie{t}^{\sigma})^*$ is the set of weights of the $\lie{t}^{\sigma}$-module $\lie{g}^{\sigma}/ \lie{t}^{\sigma}$. The inclusion $ \lie{t}^{\sigma} \hookrightarrow \lie{t}$ determines a projection map $\pi: \lie{t}^* \rightarrow (\lie{t}^{\sigma})^*$. The natural injection of $\lie{t}^{\sigma}$-modules, $\lie{g}^{\sigma}/ \lie{t}^{\sigma} \subseteq \lie{g}/\lie{t}$, implies that $ \Phi_{\sigma} \subseteq \pi(\Phi)$.

\begin{cor}\label{blah}
If the automorphism groups of $\pi(\Phi)$ and $\Phi_{\sigma}$ coincide, then $W_{\sigma}$ acts on $H^*(BLG_0^{\sigma})$ via $Out(G_0^{\sigma})$. 
\end{cor}

\begin{proof}
Since compact Lie groups can be constructed functorially from their root system and weight lattice, it will suffice to prove that the action of $W_{\sigma}$ on $T^{\sigma}$ preserves the root system $\Phi_{\sigma}$.

The action of $W_{\sigma}$ on $\lie{t}$ is a restriction of the standard Weyl group action, so it clearly preserves the root system $\Phi \subset \lie{t}^*$. The projection map $\pi: \lie{t}^* \rightarrow (\lie{t}^{\sigma})^*$ is $W_{\sigma}$-equivariant, so $W_{\sigma}$ also preserves $\pi(\Phi)$. Since by hypothesis, the automorphisms of $\pi(\Phi)$ and $\Phi_{\sigma}$ coincide, $W_{\sigma}$ must also preserve $\Phi_{\sigma}$.
\end{proof}

A sufficient condition for the hypothesis of Corollary \ref{blah} to hold is that $\pi(\Phi) = \Phi_{\sigma}$.  We have an easy-to-check criterion for this.

\begin{cor}\label{countingrootorbits}
The automorphism $\sigma$ induces a permutation of the root system $\Phi$ of $(G,T)$. If the number of $\sigma$-orbits in $\Phi$ is equal to the number of roots in $\Phi_{\sigma}$, then $W_{\sigma}$ acts on $H^*(BLG_0^{\sigma})$ via $Out(G_0^{\sigma})$. 
\end{cor}

\begin{proof}
Because we have an inclusion $ \Phi_{\sigma} \subseteq \pi(\Phi)$, it suffices to show that $\pi(\Phi) $ has cardinality equal to $\Phi_{\sigma}$. Since any two roots in $\Phi$ lying in the same $\sigma$-orbit are sent to the same element of $\pi(\Phi)$, the result follows.
\end{proof}

\begin{rmk}
We know from Lemma \ref{a torus} that  $\sigma$ preserves a Weyl chamber of $T$. Hence it also preserves a set positive roots for $\Phi$ and thus the action of $\sigma$ on roots is determined by an automorphism of the Dynkin diagram. We will use this point of view to count orbits in concrete examples in the following section.
\end{rmk}

\section{Examples}\label{examplessect}

In this section, we compute $H^*(BL_{\sigma}G;F)$ in several examples, including all automorphisms of simple Lie groups. By the following argument, it makes little difference which finite cover of the adjoint group we work with.

\begin{prop}\label{fincover}
Let $\phi: G \rightarrow G'$ be a surjective homomorphism of connected, compact Lie groups with finite kernel $K$, and let $\sigma \in Aut(G)$ descend to $\sigma' \in Aut(G')$. Then the induced map on twisted loop groups determines a cohomology isomorphism
$$ H^*(BL_{\sigma}G) \cong  H^*(BL_{\sigma'}G')$$
for coefficient fields of characteristic coprime to the order of $K$. 
\end{prop}

\begin{proof}
For any $G'$-space $X$, we may compose with $\phi$ to make $X$ into a $G$-space and resulting map of homotopy quotients 
$$ EG \times_G X \rightarrow EG' \times_{G'} X $$
has homotopy fibre $BK$.  Since $BK$ is acyclic over the coefficient field $F$, this means that $H^*_{G}(X) \cong H^*_{G'}( X)$ for any $G'$-space $X$ and in particular 
$$H^*_{G}(G'_{Ad_{\sigma'}}) \cong H^*_{G'}( G'_{Ad_{\sigma'}}).$$ 
The map $\phi$ is also a covering map with deck transformation group $K$ acting transitively on the fibres. The transfer map determines an isomorphism
$$ H^*(G)^{K} \cong H^*(G')  $$ for fields of characteristic coprime to $\# K$.  Since the deck transformations are isotopies of $G$, they act trivially on cohomology and $\phi$ is a cohomology isomorphism
$$ H^*(G) \stackrel{\phi}{\cong} H^*(G'),$$
and similarly for equivariant cohomology
$$H^*_{G }(G_{Ad_{\sigma}}) \stackrel{\phi}{\cong} H^*_{G }(G'_{Ad_{\sigma}}) $$
\end{proof}

\subsection{Untwisted loop groups }\label{untwi}

The following proposition is well known (see for example Kuribayashi, Mimura, Nishimoto \cite{kmn}, Theorem 1.2), but I am including a proof both for convenience and because I have not found a statement of the result in this generality in the literature. 

\begin{prop}\label{untwistedcase}
Let $G$ be a compact, connected Lie group. For any field $F$ of characteristic $p$ such that $H^*(G;\Z)$ is $p$-torsion free, we have 
\begin{equation}\label{tensor}
H^*(BLG;F) \cong H^*(G) \otimes H^*(BG) \cong \Lambda(x_1,...,x_r) \otimes F[y_1,...,y_r] . \end{equation} 
where $r$ equals the rank of $G$ and the degrees of the generators are independent of $F$.
\end{prop}

\begin{rmk}
A sufficient condition $H^*(G;\Z)$ to be $p$-torsion free is for $p$ to be coprime to the order of the Weyl group of $G$. We refer to Borel \cite{bl} for the degrees of the generators for various simple groups $G$.
\end{rmk}

\begin{proof}
Under the hypotheses above, we have isomorphisms $$H^*(G;F) \cong \Lambda(x_1,...,x_r),$$
an exterior algebra where $r$ and the odd degrees $\deg(x_i)$ are independent of $p$, and  $$H^*(BG;F)\cong F[y_1,...,y_r]$$  such that $\deg(y_i) = \deg(x_i)+1$ (see \cite{bl} section 9). 

The classifying space  $BLG = EG \times_G G_{Ad}$, fits into a fibration sequence
$$ G \stackrel{i}{\rightarrow} BLG \rightarrow BG$$
which has Serre spectral sequence $E_2 = H^*(BG;F) \otimes H^*(G;F) $ converging to $H^*(BLG;F)$.  This spectral sequence is known to collapse for $F= \Q $  (see \cite{b2} for example) and thus by the universal coefficient theorem must collapse for all $F$ under consideration, proving that (\ref{tensor}) holds as an isomorphism of $H^*(BG;F)$-modules. Since $H^*(G;F) = \Lambda(x_1,...,x_r)$ is free as a supercommutative algebra, we can upgrade (\ref{tensor}) to an algebra isomorphism using the Leray-Hirsch Theorem to lift the generators of $x_1,...,x_r$ via the surjection $i^*: H^*(BLG;F) \rightarrow H^*(G;F)$.
\end{proof}

\subsection{$SU(n)$ with entry-wise complex conjugation} 

Let $\sigma \in Aut(SU(n))$ denote matrix entry-wise complex conjugation. For $n \geq 3$, $[\sigma]$ generates the outer automorphism group $Out(SU(n)) \cong \Z_2$. The fixed point set $SU(n)^{\sigma} = SO(n)$. Observe that $\sigma(A)^{-1} = A^T$ where $A^T$ denotes the transpose, so the twisted adjoint action is 

\begin{equation}\label{conjugating a bilnear form}
 Ad_{\sigma}(A) ( X) = A X \sigma(A)^{-1} = A X A^T.
 \end{equation}   
 which may be interpreted as a change of basis operation for a bilinear form (see Remark \ref{bilener}).

\begin{prop}\label{refd}
For coefficient fields $F$ coprime to $n!$, the inclusion $SO(n) \hookrightarrow SU(n)$ determines an isomorphism
$$  H^*(BL_{\sigma}SU(n);F) \cong H^*(BLSO(n);F)^{\Z_2} $$ 
where $\Z_2$ acts on $SO(n)$ by an orientation reversing change of basis.
\end{prop}

\begin{proof}
For $n= 1,2$, we have $T^{\sigma} = G_0^{\sigma}$, so Proposition \ref{blah0} applies immediately. For $n \geq 3$, we must study the action of $\sigma$ on the root system $\Phi$ of $SU(n)$ in order to apply Corollary \ref{blah}. The roots of $SU(n)$ are $e_i - e_j$ for  $i,j \in \{1,...,n \}$, $i\neq j$, and the automorphism induces the involution of roots
$$\sigma( e_i -e_j)   = e_{n+1-j} -e_{n+1-i}.$$ 
Since $\sigma$ has order two, the projection map satisfies $$\pi(x) = \frac{x+ \sigma(x)}{2}.$$  

For $ i < (n+1)/2$,  define 
$$E_i = -E_{n+1-i} =  \frac{1}{2}(e_i-e_{n+1-i})  .$$ 

If $n$ is even, then for $ i\neq j$ we have
$$
 \pi(e_i -e_j)  =  \begin{cases}
 \pm E_i \pm E_j  & \text{ if $i+j \neq n+1$}\\
 \pm 2E_i  & \text { if $ i+j = n+1$}\\
                    \end{cases} 
$$
which is exactly the root system  $C_{n/2}$. If $n$ is odd, then for $ i\neq j$ we have
$$
 \pi(e_i -e_j)  =  \begin{cases}

 \pm E_i \pm E_j  & \text{ if $i+j \neq n+1$ and $\frac{n+1}{2} \not\in\{i,j\}$ }\\
  \pm E_i  & \text{ if $j = \frac{n+1}{2}$}\\
  \pm E_j  & \text{ if $i = \frac{n+1}{2}$}\\
 \pm 2E_i  & \text { if $ i+j = n+1$}.\\
                    \end{cases} 
$$
which has the same automorphism group as $C_{(n-1)/2}$. In both cases, the root system $\Phi_{\sigma}$ of $SO(n)$ is the complement of $\{\pm 2 E_i| i= 1,...,[n/2] \} $ in $\pi(\Phi)$ and the automorphism group $\Phi_{\sigma}$ agrees with that of $\pi(\Phi)$ (both being equal to the automorphism group of the root lattice of $SO(n)$).   It follows from Corollary \ref{blah} that the $W_{\sigma}$-action on $H^*(BLSO(n))$ is induced by outer automorphisms of $SO(n)$.

If $n$ is odd, the outer automorphism group of $SO(n)$ is trivial, so by Theorem \ref{thm0} we have 
$$ H^*(BL_{\sigma}SU(n)) \cong H^*(BLSO(n)).$$
Moreover, an orientation reversing change of basis must be an inner automorphism of $SO(n)$, hence also of $LSO(n)$, thus it acts by an isotopy of $BLSO(n)$ (see Segal \cite{s} section 3), so $\Z_2$ acts trivially on cohomology and the result follows.

In case $n$ is even, the outer automorphism group of $SO(n)$ is $\Z_2$ and is generated by orientation reversing change of basis.  Let $ P \in O(n)$ be an orientation reversing change of basis matrix.  Then $i P \in SU(n)$ and for any $X \in SO(n)$,
$$ P X P^{-1} = (iP) X (i P)^{-1} .$$
Thus the change of basis is induced by conjugation by an element of $SU(n)$; the result now follows from Theorem \ref{thm0}
\end{proof}

The twisted action (\ref{conjugating a bilnear form}) extends naturally to a twisted action of $U(n)$ on $U(n)$.

\begin{prop}\label{suandu}
The standard inclusion $SU(n)\hookrightarrow U(n)$ induces a cohomology isomorphism
$$ H^*(BL_{\sigma}U(n);F) \cong H^*(BL_{\sigma}SU(n);F)  $$
for coefficient fields $F$ of characteristic  coprime to both $2$ and $n$. 
\end{prop}

\begin{proof}
Consider the surjective group homomorphism $$\phi: U(1) \times SU(n) \rightarrow U(n), ~~~~ (\lambda, A) \mapsto \lambda A,$$
which has finite kernel $\Z_n$. The homomorphism is equivariant with respect to entry-wise complex conjugation, so by Lemma \ref{fincover}.

$$ H^*_{U(1) \times SU(n)}((U(1) \times SU(n))_{Ad_{\sigma}}) \cong H^*_{U(n)}(U(n)_{Ad_{\sigma}}). $$

Moreover, we have isomorphisms
\begin{eqnarray*}
H^*_{U(1) \times SU(n) }((U(1) \times SU(n))_{Ad_{\sigma}}) & \cong & H^*_{U(1) }(U(1)_{Ad_{\sigma}}) \otimes H^*_{ SU(n) }( SU(n)_{Ad_{\sigma}})\\
&\cong&  H^*(B \Z_2) \otimes H^*_{ SU(n) }( SU(n)_{Ad_{\sigma}})\\
&\cong&  H^*_{ SU(n) }( SU(n)_{Ad_{\sigma}})\\
\end{eqnarray*} 
because the twisted $U(1)$-action on $U(1)_{Ad_{\sigma}}$ is transitive with stabilizer  $\Z_2$ and $B\Z_2$ is acyclic over $F$.
\end{proof}

\begin{cor}\label{[rcoelr2}
Let $F$ be a field of characteristic coprime to $n!$ and let $n =2m$ or $n =2m+1$. The standard inclusion of groups $$ LO(n) \hookrightarrow  L_{\sigma} U(n) \hookleftarrow L_{\sigma}SU(n)$$ induces isomorphisms
\begin{eqnarray*}  
H^*(BLO(n);F) & \cong & H^*(BL_{\sigma}U(n);F) \cong H^*(BL_{\sigma}SU(n);F) \\ & \cong & \Lambda(x_3, x_7,...,x_{4m -1}) \otimes S(y_4,...,y_{4m}),
\end{eqnarray*}
where the subscripts indicate the degrees of the generators.
\end{cor}

\begin{proof}
The loop group $LSO(n)$ sits inside $LO(n)$ as an index two subgroup, so  $ H^*(BLO(n)) \cong H^*(BLSO(n))^{\Z_2} $.  The result now follows from Propositions \ref{suandu} and \ref{refd}.
\end{proof}

\begin{rmk}
Corollary  \ref{[rcoelr2} stands in contrast with the formula 
$$ H^*(BL_{\sigma}U(n);\Z_2) \cong  H^*(BLU(n);\Z_2) \cong   \Lambda(x_1, x_3,...,x_{2n -1}) \otimes S(y_2, y_4,...,y_{2n}) $$
derived in Baird \cite{b3}.
\end{rmk}

\begin{rmk}\label{bilener}
The twisted action of $U(n)$ on $U(n)$  is homotopy equivalent to the change of basis action of $GL_n(\C)$ on the space of (not necessarily symmetric) non-degenerate bilinear forms on $\C^n$.  Thus Corollary \ref{[rcoelr2} also calculates the cohomology of the topological moduli stack of rank $n$ non-degenerate bilinear forms over $\C$.
\end{rmk}

\subsection{$SO(2n)$ with orientation reversing change of basis}

The Weyl group of $SO(2n)$ has order $2^{n-1}n!$. An orientation reversing change of basis determines an automorphism $\sigma \in Aut(SO(2n))$ of order two, which generates $Out(SO(2n))$ for $n \geq 5$. The corresponding twisted loop group $L_{\sigma}SO(2n)$ can be understood as the gauge group of orthogonal, orientation preserving gauge transformations of a non-orientable $\R^n$-bundle over $S^1$.

\begin{prop}\label{[rcoelr}
The block sum inclusion $SO(2n-1) \hookrightarrow SO(2n)$ induces a cohomology isomorphism 
$$H^*(BL_{\sigma}SO(2n);F) \cong H^*(BLSO(2n-1);F) \cong \Lambda(x_3, x_7,...,x_{4n -5}) \otimes S(y_4,...,y_{4n -4})$$ 
for coefficient field $F$ of odd characteristic cop rime to $n!$.
\end{prop}

\begin{proof}
Set $$\sigma(A) = PAP^{-1}$$ where $$P = 	\left( \begin{array}{cc}
Id_{2n-1} & 0   \\
0 & -1  \end{array} \right).$$   Then $SO(2n)^{\sigma}$ is isomorphic to $O(2n-1)$ by the injection
\begin{equation}\label{injection}
O(2n-1) \hookrightarrow SO(2n),~~~ B \mapsto 	\left( \begin{array}{cc}
B & 0   \\
0 & \det(B)  \end{array} \right).\end{equation}
which has identity component isomorphic to $SO(2n-1)$. 

For $n =1$, we have $SO(2) \cong U(1)$ so this case has already been covered by Corollary \ref{[rcoelr2}.  

For $n \geq 2$, the root system of $SO(2n)$ consists of vectors $\pm (e_i \pm e_j)$ for $i\neq j$ in $\{1,...,n\}$. The involution fixes $e_i$ for $ i< n$ and sends $e_n$ to $-e_{n}$. One easily checks that there are $ 4 {n \choose 2 } -2(n-1) = 2 (n-1)^2$ which equals the number of roots of the fixed point subgroup $SO(2n-1)$. Thus by Corollary \ref{countingrootorbits}, $W_{\sigma}$-action on $H^*(BLSO(2n-1))$ is induced by outer automorphisms of $SO(2n-1)$. The outer automorphism group of $SO(2n-1)$ is trivial, so the result follows by Theorem \ref{thm0}.
\end{proof}

\subsection{$SO(8)$ with the triality automorphism}

The outer automorphism group of $SO(8)$ is the permutation group $S_3$. The order three automorphisms are represented by the triality automorphisms $\sigma$, $\sigma^2 \in Aut(SO(8))$ which are related to realization of $SO(8)$ as orthogonal transformations of the underlying vector space of the octonions (see Baez \cite{b}).  The fixed point set $SO(8)^{\sigma}$ is equal to the automorphism group of the octonions $G_2$.  

The root system of $SO(8)$ has $12$ positive roots.  The triality automorphism determines 6 orbits: three of order one and three of order three.   Since the root system of $G_2$ has six positive roots, Corollary \ref{countingrootorbits} implies that  $W_{\sigma}$-acts via outer automorphisms of $G_2$.  Since $G_2$ has trivial outer automorphism group and they Weyl group of $SO(8)$ has order $3\cdot 2^6$, we conclude from Theorem \ref{thm0} that

\begin{prop}
For $\sigma \in Aut(SO(8))$ a triality automorphism we have 
 $$H^*(BL_{\sigma} SO(8);F) \cong H^*(BL_{\sigma^2} SO(8);F) \cong H^*(BLG_2;F) \cong \Lambda(x_3,x_{11}) \otimes F[y_4,y_{12}].$$ for coefficient field $F$ of characteristic coprime to $6$.
\end{prop}

\subsection{$E_6$ with involution}

The outer automorphism group of a compact, simply connected group of type $E_6$ is generated by an automorphism $\sigma$ of order two.  The induced action on the set of positive roots has 24 orbits.   The fixed point set $E_6^{\sigma}$ is isomorphic to $F_4$ which has 24 positive roots, so Corollary \ref{countingrootorbits} applies. The Weyl group of $E_6$ has order $51840=2^73^45$. Since  $F_4$ has trivial outer automorphism group we get from Theorem \ref{thm0} that.

\begin{prop}
For $\sigma \in Aut(E_6)$ not an inner automorphism, we have a cohomology isomorphism
$$ H^*(BL_{\sigma}E_6;F) \cong H^*(BLF_4; F) \cong \Lambda(x_3,x_{11},x_{15},x_{23}) \otimes F[y_4,y_{12}, y_{16},y_{24}]$$
for coefficient fields $F$ of characteristic greater to $30$.
\end{prop}

\end{document}